\documentclass[12pt,final]{article}

\usepackage[paper=letterpaper,margin=1in,twoside=false,includehead,headheight=18pt]{geometry}

\usepackage{todonotes}

\usepackage[utf8]{inputenc}

\usepackage{amsmath,amssymb,amsthm} 
\usepackage{thmtools,thm-restate}
\usepackage{graphicx}
\usepackage{showkeys}
\usepackage{mathtools}
\usepackage{booktabs}
\usepackage{tabularx}
\usepackage{enumitem}
\usepackage{textcomp}
\usepackage{cancel}
\usepackage{xparse}

\usepackage{pgfplots}

\pgfplotsset{compat=1.14}

\pgfplotsset{tick label style={font=\scriptsize},ticklabel shift=-3pt}
\pgfplotsset{mathy/.style={minor tick num=1, ticks=major, axis x line=middle, axis y line=middle, grid=both}}

\usepackage[draft]{fixme}

\newtheorem{theorem}{Theorem}[section]
\newtheorem{lemma}[theorem]{Lemma} 
\newtheorem{proposition}[theorem]{Proposition} 
\newtheorem{corollary}[theorem]{Corollary}

\newtheorem*{main}{Main Theorem}

{\theoremstyle{definition} 
 
\newtheorem*{remark}{Remark}

 }

\newtheoremstyle{named}%
  {}{}						
  {\upshape}				
  {0pt}{\bfseries}			
  {.}						
  {.5em}					
  {\thmname{#1}\thmnote{ #3}}  

\theoremstyle{named}

\newcommand{\R}{\mathbb{R}}
\newcommand{\N}{\mathbb{N}}

\renewcommand{\d}{\delta}

\def\indefinite#1d#2{{\int #1 \, d#2}}
\def\definite#1d#2#3#4{{\int_{#3}^{#4} {#1} \, d#2}}
\def\dby#1#2{\frac{d{#1}}{d{#2}}}

\DeclarePairedDelimiter\floor{\lfloor}{\rfloor}

\DeclarePairedDelimiter\parens{(}{)}

\DeclarePairedDelimiterX\setof[2]{\{}{\}}{#1\,:\,#2}

\newcommand{\disp}{\displaystyle}

\renewcommand{\k}[1]{k_{#1}}
\newcommand{\kt}{\k{t}}
\newcommand{\ft}{\tilde{f}_{r+1,t-1}}

\newcommand{\s}[1]{s_{#1}}
\newcommand{\sr}{\s{r}}

\newcommand{\e}{\varepsilon}
\renewcommand{\l}{\ell}

\newcommand{\on}{\mathop{\circ}}
\newcommand{\inv}{^{-1}}

\renewcommand{\a}{\alpha}
\renewcommand{\b}{\beta}
\newcommand{\phimin}{\phi_{\text{min}}}
\newcommand{\pk}{\phi_{\text{key}}}

\newcommand{\Lab}[1][a,b]{L_{#1}}
\newcommand{\Lr}{\Lab[r-1,1]}
\newcommand{\Fab}[1][a,b]{F_{#1}}

\DeclareMathOperator{\ex}{ex}

\begin{document}

\title{Supersaturation for subgraph counts}
\date{\today}
\author{Jonathan Cutler\thanks{Supported in part by the Simons Foundation under Grant No.~524435.}\\
\small{Montclair State University}\\
\small{\texttt{jonathan.cutler@montclair.edu}}
 \and 
JD Nir\\
\small{University of Manitoba}\\
\small{\texttt{jd.nir@umanitoba.ca}}
 \and
A.~J.~Radcliffe\thanks{Supported in part by the Simons Foundation under Grant No.~429383.}\\
\small{University of Nebraska-Lincoln}\\
\small{\texttt{jamie.radcliffe@unl.edu}}}
\maketitle

\begin{abstract}
	The classical extremal problem is that of computing the maximum number of edges in an $F$-free graph.  In particular, Tur\'an's theorem entirely resolves the case where $F=K_{r+1}$.  Later results, known as supersaturation theorems, proved that in a graph containing more edges than the extremal number, there must also be many copies of $K_{r+1}$.  
	
	Alon and Shikhelman introduced a broader class of extremal problems, asking for the maximum number of copies of a graph $T$ in an $F$-free graph (so that $T=K_2$ is the classical extremal number).  In this paper we determine some of these generalized extremal numbers when $T$ and $F$ are stars or cliques and prove some supersaturation results for them.

\end{abstract}

\section{Introduction} 
\label{sec:introduction}

The classic theorem of Tur\'an \cite{T} gives the maximum number of edges in a $K_{r+1}$-free graph, a number which is asymptotically $(1-\frac{1}r)\binom{n}2$.  As is standard, we let $\ex(n,F)$ be the maximum number of edges in an $F$-free graph on $n$ vertices, so Tur\'an's theorem determines $\ex(n,K_{r+1})$.  Of course, if the number of edges in a graph $G$ on $n$ vertices exceeds $\ex(n,K_{r+1})$, we know that $G$ must contain at least one $K_{r+1}$.  One can ask about the minimum number of copies of $K_{r+1}$ that are contained in $G$.  Results of this type are referred to as \emph{supersaturation} theorems.  To be precise, letting $k_{r+1}(G)$ be the number of copies of $K_{r+1}$ in a graph $G$, supersaturation questions ask one to determine
\[
	\min \setof{k_{r+1}(G)}{\text{$G$ a graph with $n$ vertices and $\ex(n,K_{r+1})+q$ edges}},
\]
for some $q\geq 1$.  When $q=o(n^2)$, the problem was studied by Rademacher \cite{R}, Erd\H{o}s \cite{E,E1,E2}, and then resolved by Lov\'asz and Simonovits \cite{LS,LS1}.  For the case when $q=\Omega(n^2)$, asymptotic solutions have been found by Razborov \cite{Rz} for $r=2$, Nikiforov \cite{N1} for $r=3$, and Reiher \cite{Rei} for general $r$.  See Pikhurko and Yilma \cite{PY} for a very informative introduction to supersaturation.

One could also ask if other structures are guaranteed to exist in graphs with more edges than the Tur\'an number.  The following theorem of Erd\H{o}s and Stone \cite{ES} shows that, in a graph in which the edge count exceeds this extremal number by a constant multiple of $n^2$, must not only contain a $K_{r+1}$, but indeed a blowup of $K_{r+1}$ with large part sizes.  For a graph $G$, we let the blowup $G(b)$ be the graph where each vertex of $G$ is replaced by an independent set of size $b$ and each edge is replaced by a complete bipartite graph.    

\begin{theorem}[Erd\H{o}s-Stone]\label{thm:es}
    Let $r\geq 1$ be an integer and let $\e>0$.  Then there exists $n_0=n_0(r,\e)$ such that if $G$ is a graph on $n\geq n_0$ vertices and
    \[
        e(G)\geq \parens[\Big]{1-\frac{1}r+\e}\binom{n}2,
    \]
    then $G$ contains $K_{r+1}(b)$ for some $b\geq \e\log n/(2^{r+1}(r-1)!)$.
\end{theorem}

\noindent We will refer to theorems analogous to Theorem~\ref{thm:es} as \emph{structural supersaturation} results.

In more recent work, Alon and Shikhelman \cite{AS} considered  generalized extremal problems involving counting copies of some fixed subgraph rather than edges.  To be precise, they were interested in determining values of 
\[
    \ex_T(n,F)=\max \setof{n_T(G)}{\text{$G$ is an $F$-free graph on $n$ vertices}},
\]
where $n_T(G)$ is the number of copies of $T$ in $G$. In particular,
$\ex_{K_2}(n,F)=\ex(n,F)$. Their paper lead to many investigations by various
authors for different choices of $T$ and $F$; see \cite{GP,GSTZ,HP,LM,L} for a sample of
authors and results. In this paper, we consider Alon-Shikhelman-type problems where $T$ and $F$ are either cliques or stars.  We also consider supersaturation and structural supersaturation results in this vein.  The following subsections will outline the history of these problems and the new results of this paper.

\subsection{Cliques without cliques} 
\label{sub:cliques_versus_cliques}

The most fundamental Alon-Shikhelman-type problems involve cliques.  As above,
we write $\kt(G)$ for $n_{K_t}(G)$.  Zykov \cite{Z}, along with many others,
showed that $\ex_{K_t}(n, K_{r+1}) = k_t(T(n, r))$, where $T(n, r)$ is the
$r$-partite Tur\'an graph on $n$ vertices.  Bollob\'{a}s \cite{B} discussed the
general problem of minimizing the number of copies of $K_{s}$ in a graph with a
given number, say $N$, of copies of $K_t$, i.e., a supersaturation result for
$\ex_{K_t}(n,K_{s})$.  (Of course, if $N\leq \ex_{K_t}(n,K_{s})$, then this
minimum number is $0$.)  His result gives a bound of the form 
\[
	k_s(G)\geq \psi(N),
\]
where $\psi$ is a function defined implicitly.  
\begin{theorem}[Bollob\'as]\label{thm:bb}
 	For a given $n$, let $\psi(x)=\psi_t^s(x)$ be the maximal convex function defined for $0\leq x\leq \binom{n}t$ such that
 	\[
 		\psi\parens[\bigg]{\parens[\Big]{\frac{n}i}^t\binom{i}t}\leq \parens[\Big]{\frac{n}i}^s\binom{i}s, \qquad \text{for $i=1,2,\ldots,n$}.
 	\]
 	Then, if $G$ is a graph on $n$ vertices with $k_t(G)\geq x$, then $k_s(G)\geq \psi(x)$.
\end{theorem}
A weaker, but slightly more transparent version is the following.

\begin{theorem}\label{thm:nir}
	Let $\theta$ be a real number and $s$ and $t$ be integers with $2\leq t\leq s\leq \theta+1$.  If $G$ is a graph on $n$ vertices such that $k_t(G)\geq \binom{\theta}{t}\parens{n/\theta}^t$, then $k_s(G)\geq \binom{\theta}{s}\parens{n/\theta}^s$.
\end{theorem}

We cannot find this result in the literature, but it can be proved by iteratively applying the following beautiful theorem of Moon and Moser \cite{MM} and the method outlined in Lov\'asz's Combinatorial Problems and Exercises \cite[Section 10, Question 40]{LL}.

\begin{theorem}[Moon-Moser]
	For any graph $G$ on $n$ vertices and any $s\ge 2$,
	\[
		\frac{k_{s+1}(G)}{k_s(G)} \ge \frac1{s^2-1}\parens[\Big]{s^2\frac{k_s(G)}{k_{s-1}(G)}-n}.
	\]
\end{theorem}

Both Theorem~\ref{thm:bb} and Theorem~\ref{thm:nir} are supersaturation results.  Now we turn our attention to the corresponding structural supersaturation problem.  Nikiforov \cite{N} showed that the conclusion of the Erd\H{o}s-Stone theorem follows even from the weak hypothesis that $G$ contains $\Omega(n^{r+1})$ copies of $K_{r+1}$.

\begin{theorem}[Nikiforov]\label{thm:niki}
	Let $s\geq 2$ and $c$ and $n$ be such that 
	\[
		0<c<1/s!\quad\text{and}\quad n\geq \exp(c^{-s}).
	\]
	If $G$ is a graph with $n$ vertices and $k_s(G)\geq cn^{s}$, then $G$ contains a $K_s(b)$ with $b=\floor{c^s\log n}$. 
\end{theorem}

This, together with Theorem~\ref{thm:nir}, proves a structural supersaturation extension of Zykov's result.

\begin{theorem}
	For all $\e>0$, there is a $\d>0$ and an $n_0\in \N$ such that if $G$ is a graph on $n\geq n_0$ vertices and $k_t(G)\geq (1+\e)k_t(T(n,r))$, then $G$ contains a $K_{r+1}(C\log n)$ for some $C=C(\e,r)>0$. 
\end{theorem}

\begin{proof}
	The hypothesis on $k_t(G)$ implies that, for some $\theta>r$, we have $k_t(G)\geq \binom{\theta}{t}\parens{n/\theta}^t$.  Thus, by Corollary~\ref{thm:nir}, $k_{r+1}(G)\geq \binom{\theta}{r+1}\parens{n/\theta}^{r+1}$, a constant multiple of $n^{r+1}$.  Now, by Theorem~\ref{thm:niki}, $G$ contains a large blowup of $K_{r+1}$.
\end{proof}


\subsection{Cliques without stars} 
\label{sub:cliques_versus_stars}

If we write $S_{r}$ for $K_{1,r}$, a recent result of Chase \cite{C}, building on work of Gan, Loh, and Sudakov \cite{GLS}, completely determines $\ex_{K_t}(n,S_{r+1})$.

\begin{theorem}[Chase]\label{thm:chase}
	Fix $t\geq 3$.  For any positive integers $n,r\geq 1$, if $n=a(r+1)+b$ where $0\leq b\leq r$, then
	\[
		\ex_{K_t}(n,S_{r+1})=a\binom{r+1}t+\binom{b}t.
	\]
\end{theorem}

It will be useful for us later to state and prove here a ``signpost'' version of Theorem~\ref{thm:chase}
due to Wood \cite{W}, and Engbers and Galvin \cite{EG}.  For $v\in V(G)$, we write $\kt(v)$ for the number of copies of $K_t$ in $G$ that contain vertex $v$.

\begin{theorem}[Wood, Engbers-Galvin]\label{thm:one}
   For any $1\leq r\leq n$, we have 
    \[
        \ex_{K_t}(n,S_{r+1})\leq  \frac{n}{t} \, \binom{r}{t-1} =\frac{n}{r+1} \, \binom{r+1}{t}.
    \]
\end{theorem}
\begin{proof}
    Note that being $S_{r+1}$-free is equivalent to having maximum degree at most $r$.  Let $G$ be such a graph on $n$ vertices.  If we count pairs $(v,S)$ where $v$ is a vertex of $G$, $S$ is a $t$-clique in $G$ and $v\in S$ then 
    \[
        t \kt(G) = \sum_{v\in V(G)} \kt(v) = \sum_{v\in V(G)} \k{t-1}(G[N(v)])
\le n \binom{r}{t-1}. 
    \]
\end{proof}
Note that though Theorem~\ref{thm:one} does not give the exact value of $\ex_{K_t}(n,S_{r+1})$, it is asymptotically sharp since the graph $aK_{r+1}$ achieves the bound whenever $n$ is divisible by $r+1$.

In Section~\ref{sec:es}, we prove the following supersaturation result showing that if $G$ contains many copies of $K_t$, then there must be many copies of $S_r$ in $G$. We write $\sr(G)$ for $n_{S_r}(G)$ and note that
\[
    \sr(G) = \sum_{v\in V(G)} \binom{d(v)}r.
\]

\begin{restatable}{theorem}{main}\label{thm:esy}
    Given $2\le t\le r$, for all $\e>0$ there exists $\d>0$ such that if $G$ is a graph on $n$ vertices having
    \[
        \kt(G) \ge (1+\e) \frac{n}{r+1} \, \binom{r+1}t,
    \]
    then $\s{r+1}(G) \ge \d n$.
\end{restatable}

	Note that the bound in Theorem~\ref{thm:esy} is asymptotically sharp.  To see this, let $s>r$ and consider the graph $G$ on $n=k(s+1)$ vertices that is the disjoint union of $k$ copies of $K_{s+1}$, i.e., $G=kK_{s+1}$.  Then, 
	\[
		k_t(G)=k\binom{s+1}t=\frac{n}t\binom{s}{t-1}.
	\]
A straightforward calculation shows that, provided $s\geq r+1$,
\[
	\frac{\frac{n}t\binom{s}{t-1}}{\frac{n}{r+1}\binom{r+1}t}>1,
\]
and so the conditions of Theorem~\ref{thm:esy} are met.  Further, note that 
\[
	s_{r+1}(G)=n\binom{s}{r+1}.
\]
Thus, equality is achieved in the conclusion of Theorem~\ref{thm:esy} with $\d=\binom{s}{r+1}$.

For structural supersaturation, since the star is not vertex transitive, there are different notions of a blowup of $S_{r+1}$; they are all of the form $K_{a,b}$.  The discussion above implies that having a surplus of $K_t$s does not imply even the existence of a $K_{1,r+2}$.  In addition, the classic construction of F\"uredi \cite{F} demonstrates that it is also not possible to guarantee the existence of a $K_{2,r+1}$ (at least in the case when $t=3$).  The following theorem can be read out of his paper.

\begin{theorem}[F\"{u}redi]
	For any $r\geq 1$ there exist infinitely many $n$ so that there is a graph on $n$ vertices which is $K_{2,r+1}$-free and contains $\Omega(n^{3/2})$ triangles.  In particular, knowing that $k_3(G)$ is at least $(1+\e)\ex_{K_3}(n,S_{r+1})$ does not imply the existence of a $K_{2,r+1}$ in $G$.
\end{theorem}


\subsection{Stars without stars} 
\label{sub:stars_versus_stars}

Although this case is rather uninteresting, we include it for completeness. 

\begin{proposition}
	If $t>1$, then for $n\geq r+1$,
	\[
		\ex_{S_t}(n,S_{r+1})=\begin{cases}
					n\binom{r}t & \text{if $nr$ is even},\\
					(n-1)\binom{r}t+\binom{r-1}t & \text{otherwise}.
					\end{cases}
	\]
\end{proposition}

\begin{proof}
	Since each degree is at most $r$, we have that $s_t(G)=\sum \binom{d(v)}t$ is maximized when $G$ is as close to $r$-regular as possible.  If $nr$ is even, then there is an $r$-regular graph and otherwise there is a graph where one vertex has degree $r-1$ and all others have degree $r$.
\end{proof}

One can also prove a rather uninteresting supersaturation result in this case.  Since both the number of $S_t$s and the number of $S_{r+1}$s are a function of the degree sequence, it is easy to check that an excess of $\e n\binom{r}t$ copies of $S_t$ yields at least $\e n(r-t+1)/t$ copies of $S_{r+1}$.  The extremal graph is as regular as possible. 

No structural supersaturation theorem for this case is true. Any $(r+1)$-regular graph has a fixed fraction more $S_t$s than $\ex_{S_t}(n,S_{r+1})$, without containing any $S_{r+2}$.  The same F\"uredi example from the previous section is almost regular and hence contains at least $(1+\e)\ex_{S_t}(n,S_{r+1})$ copies of $S_t$ without having a $K_{2,r+1}$.


\subsection{Stars without cliques} 
\label{sub:stars_without_cliques}


This case is substantially more difficult than the others we've encountered up
to this point. Caro and Yuster \cite{CY}
considered the related problem of determining, for a graph $H$ and $t\geq 1$,
\[
	\ex_t(n,H)=\max\setof[\Big]{f_t(G)}{\text{$G$ is $H$-free on $n$ vertices}},
\]
where
\[
	f_t(G)=\sum_{v\in V(G)} d_G(v)^t.
\]
Note that the values of $\ex_t(n,K_{r+1})$ and $\ex_{S_t}(n,K_{r+1})$ are
asymptotically equal as $n\to \infty$. Caro and Yuster showed that the extremal
graph for $H=K_{r+1}$ and $t=1,2,3$ is the Tur\'an graph, and asked if this was
true for larger $t$. Bollob\'as and Nikiforov \cite{BN} gave a nearly complete answer. We sum up their results in the following theorem:

\begin{theorem}[Bollob\'as, Nikiforov]\label{thm:bn_result}
For all $r \ge 2$ and $t > 0$ there exists $c = c(t,r)$ such that if some
$K_{r+1}$-free graph $G$ of order $n$ satisfies $f_t(G) = \ex_t(n,K_{r+1})$, then $G$ is
a complete $r$-partite graph having $r-1$ vertex classes of size $cn+o(n)$.
Furthermore, if $t < r$ and $n$ is sufficiently large, then the Tur\'an graph
$T_r(n)$ realizes $\ex_t(n,K_{r+1})$, but for $t \ge r+\sqrt{2r}$, $\ex_t(n,K_{r+1}) >
f(t,T_r(n))$.
\end{theorem}

Considering the graphon version $\ex_{S_t}(W,K_{r+1})$ of the stars without
cliques problem, we extend the Bollob\'as-Nikiforov result in two ways. First,
we specify more precisely the sizes of the vertex classes of non-Tur\'an
solutions. Second, we prove the existence of a value $t^\ast(r)$ such that when
$t < t^\ast$, the Tur\'an graphon uniquely realizes $\ex_{S_t}(W,K_{r+1})$ and
when $t \ge t^\ast$, the non-Tur\'an solution is the unique maximum. We have not
been able to make any progress on either the supersaturation or structural
supersaturation versions of the problem. All the details can be found in
Section~\ref{sec:many_stars}.



\section{Supersaturation for cliques without stars} 
\label{sec:es}

In order to prove Theorem~\ref{thm:esy}, we start with a lemma concerning the function $\binom{x}s$ where $x$ is a postive real number.
    We define, for $x\in [0,\infty)$ and $s\in \N_{\ge 1}$,
    \[
        f_s(x) =\begin{cases} 
		\displaystyle \binom{x}s = \frac1{s!} \, x(x-1)\cdots(x-s+1) & \text{if $x\geq s-1$}\\
		\quad 0 & \text{if $0\leq x<s-1$}.
		\end{cases}
    \]

Note that, for $x>s+1$,
\[
	f'_s(x)=\frac{1}{s!}\sum_{i=0}^{s-1} x(x-1)\cdots\bcancel{(x-i)}\cdots (x-s+1)
\]
and
\[
	f''_s(x)=\frac{2}{s!}\sum_{0\leq i<j\leq s-1} x(x-1)\cdots\bcancel{(x-i)}\cdots\bcancel{(x-j)}\cdots (x-s+1).
\]
Also, note that $f_s$ is strictly increasing on $[s-1,\infty)$.  We denote the inverse of $f_s|_{[s-1,\infty)}$ by $f_s^{-1}$.

\begin{lemma}\label{ref:conv}
    For all $1\le t < s$ the function $f_s \on f_t\inv$ is convex on $(0,\infty)$ and strictly convex on $\parens[big]{\binom{s-1}t,\infty}$.
\end{lemma}

\begin{proof}
Note that $f_s\on f_t\inv(x)=0$ if $x\leq \binom{s-1}t$. Further, the derivative is positive if $x>\binom{s-1}t$ and thus it's enough to show strict convexity on $\parens[big]{\binom{s-1}t,\infty}$.
    For convenience we'll denote $f_t\inv(x)$ by $u$, and we may assume that $u>s-1$. Note that 
\[
    \parens{f_s \on f_t\inv}'(x) = f_s'(u) \cdot u', \qquad\text{and}\qquad u' = \frac{1}{f_t'(u)} . 
\]
Thus 
\begin{align*}
    \parens{f_s\on f_t\inv}'' &= f_s''(u) \cdot \parens{u'}^2 + f_s'(u)\cdot u'' 
        = f_s''(u)\cdot \frac{1}{\parens{f_t'(u)}^2} - \frac{f'_s(u)}{\parens{f_t'(u)}^2} \cdot f_t''(u)\cdot u' \\
        &= \frac{f_s''(u)f_t'(u) - f_s'(u)f_t''(u)}{\parens{f_t'(u)}^3}.
\end{align*}
Since $u>t-1$, we have $f'_t(u)>0$ so we need only that the numerator of the above is positive.  To this end, since $s>t$, note that
\begin{align*}
	f''_s(u)f'_t(u)&-f'_s(u)f''_t(u)\\
	&=\frac{2}{s!t!}\Biggl[\sum_{\substack{0\leq i<j\leq s-1\\ 0\leq k\leq t-1}} u(u-1)\cdots\bcancel{(u-i)}\cdots\bcancel{(u-j)}\cdots(u-s+1)\\
	&\hspace{2.25in}\cdot u(u-1)\cdots \bcancel{(u-k)}\cdots (u-t+1)\\
	&\\
	&\qquad-\sum_{\substack{0\leq i \leq s-1\\ 0\leq j<k\leq t-1}} u(u-1)\cdots\bcancel{(u-i)}\cdots (u-s+1)\\
	&\hspace{2.25in}\cdot u(u-1)\cdots\bcancel{(u-j)}\cdots\bcancel{(u-k)}\cdots (u-t+1)\Biggr].
\end{align*}
We'll show that this is non-negative by proving that all the negative terms are
canceled by positive ones.  If we write $T_{ij|k}$ for a typical term in the
first sum and $T_{i|jk}$ for one in the second, then we see that all the terms
with $i,j,k< t$ cancel since $T_{i|jk}$ cancels with $T_{jk|i}$.  The remaining
negative terms are of the form $T_{i|jk}$ with $i\geq t$.  We have that each
such term $T_{i|jk}$ cancels with $T_{ji|k}$.  Strictness of convexity is
guaranteed since some strictly positive terms remain, e.g., the $T_{ij|k}$ with
$i=k$ and $j\geq t$. 
\end{proof}

We are now ready for the proof of the main theorem of this section, which we recall here.

\main*

\begin{proof}
	As in Theorem~\ref{thm:one},
	\[
		n\binom{r}{t-1}(1+\e)\leq t\kt(G)
		=\sum_v \kt(v)
		\leq \sum_v \binom{d(v)}{t-1}.
	\]
	Define $\l(v)=f_{t-1}(d(v))=\binom{d(v)}{t-1}$.  We have
	\[
		\sum_v \l(v)\geq n\binom{r}{t-1}(1+\e)\qquad\text{and}\qquad s_{r+1}=\sum_v \binom{d(v)}{r+1}=\sum_v \binom{f^{-1}_{t-1}(\l(v))}{r+1}.
	\]
	The last equality is true term-by-term noting that if $d(v)<t-1$, and hence $d(v)\neq f_{t-1}^{-1}(f_{t-1}(d(v)))$, the $v$ term in both these sums is zero.  
	
	We define 
	\[
		\ft(\l)=f_{r+1}(f_{t-1}^{-1}(\l)).
	\]
	We will determine the minimum of $\sum_{i=1}^n \ft(\l_i)$ subject to $\sum_{i=1}^n \l_i\geq n\binom{r}{t-1}(1+\e)$.  To be precise, we solve the relaxation where $\l_i\in \R_{\geq 0}$.  Since $\ft$ is convex by Lemma~\ref{ref:conv}, we have
	\[
		\sum_{i=1}^n \ft(\l_i)\geq n\ft\Bigl(\sum_{i=1}^n \l_i\Bigr)\geq n\ft\Bigl(\binom{r}{t-1}(1+\e)\Bigr).
	\]
	Thus we are done, setting $\d=\ft\Bigl(\binom{r}{t-1}(1+\e)\Bigr)$.
\end{proof}


\section{Many stars, no $K_{r+1}$} 
\label{sec:many_stars}

%


 We first note that any extremal graph is complete $r$-partite. This simplifies the problem of finding the largest number of $S_t$s. We then address the graphon version of the problem, determining $\ex_{S_t}(W,K_{r+1})=\ex_t(W,K_{r+1})$.
Erd\H{o}s \cite{E64} proved that, given any $K_{r+1}$-free graph $G$, there is an $r$-partite graph $H$ on the same vertex set satisfying $d_G(v)\leq d_H(v)$ for all vertices $v$.  Hence, there is an $r$-partite optimizer for $\ex_{S_t}(n,K_{r+1})$.  Indeed, Gy\H{o}ri, Pach, and Simonovits \cite{GPS} proved the following substantially stronger result.

\begin{theorem}[Gy\H{o}ri-Pach-Simonovits] \label{thm:gps}
Let $T$ be a complete $k$-partite graph with $t+1$ vertices, let $r \ge k$, and let $n \ge \max(t+2,r+1)$. Then all $K_{r+1}$-free graphs $G$ on $n$ vertices satisfying $n_T(G) = \ex_{T}(n,K_{r+1})$ are complete $r$-partite.
\end{theorem}

\begin{proof}[Sketch of proof]
For a graph $G$ and non-adjacent vertices $u,v \in V(G)$, the \emph{Zykov
symmetrization} of $v$ by $u$, denoted $Z_{u\to v}(G)$, turns $v$ into a clone
of $u$ by setting $N(v) = N(u)$. This technique was introduced by Zykov in
\cite{Z}. Note that Zykov symmetrizations cannot increase the clique number of a
graph, and indeed, one can be chosen that (weakly) increases the number of
copies of $T$ until $G$ is complete multipartite.  (That $G$ becomes complete
multipartite is only guaranteed because $T$ is complete multipartite itself.) At
any cloning step where the graph becomes complete multipartite, one can show
that the graph is transformed from a strict subgraph of a complete multipartite
graph to a complete multipartite graph, strictly increasing the number of copies
of $T$ in the process. 
\end{proof}

Knowing that the optimal graph is complete multipartite leaves only the question of what part sizes are optimal. We solve the problem asymptotically, i.e., we show that there are optimal proportions $\a_1, \a_2, \dots, \a_{r}$ for the part sizes. 
The optimization problem we are trying to solve then is (asymptotically, and ignoring a factor of $1/t!$)
\begin{equation}
    \begin{array}{lrcl}
        \text{Maximize}     & F(\rho_1, \rho_2, \dots, \rho_{r}) &=& \disp \sum_{i=1}^r \rho_i (1-\rho_i)^t \\
        \text{subject to}   & \rho_i &\ge& 0 \\
                            & \disp \sum_{i=1}^r \rho_i &=& 1 .
    \end{array} \label{opt}
\end{equation}

We will naturally start by finding the interior critical points, which must satisfy 
\[
    \nabla F(\rho) =  \lambda (1,1,\dots,1)
\]
for some $\lambda$. Writing $f(\rho) = (1-\rho)^t \rho$ we require that the
vector  $(f'(\rho_1),f'(\rho_2),\dots,f'(\rho_r)) $ is constant. We start with a
basic lemma concerning the derivatives of $f$.

\begin{figure}[ht]
    \begin{center}
        \begin{tikzpicture}
                \begin{axis}[%
                    mathy,
                    xmin=0, xmax=1,
                    ymin=-0.25, ymax=1,
                    width=4in]
                
					\addplot[red, thick, domain=0:1, samples=100,
							  declare function={
							  	t=6;
								f(\x)=(1-\x)^t*\x;
							  }
							]
						{f(x)};
					\addlegendentry{$f(\rho)$};
					
					\addplot[blue, thick, domain=0:1, samples=100,
							 declare function={
							   	t=6;
							   	g(\x)=(1-\x)^(t-1)*(1-(t+1)*\x);
							  }
							]
						{g(x)};
					\addlegendentry{$g(\rho)$};
                \end{axis}
            \end{tikzpicture}
    \end{center}
    \caption{Graphs of $f(\rho)$ and $g(\rho)$ with $t=6$}
    \label{fig:lines}
\end{figure}
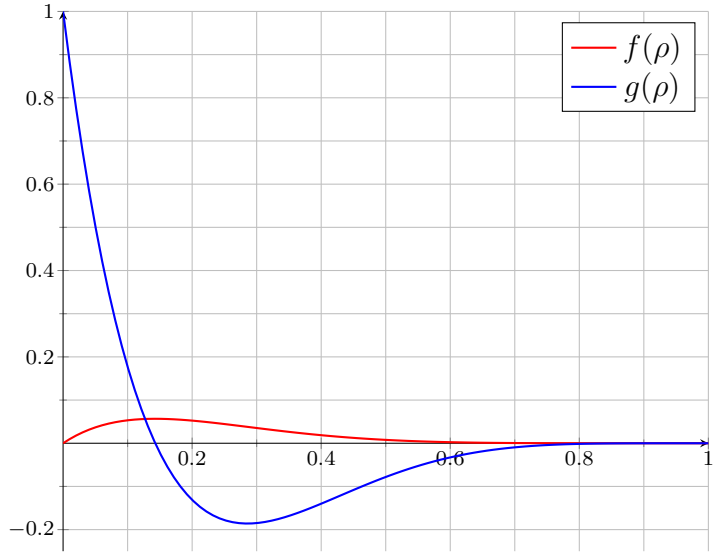

\begin{lemma}\label{lem:g_facts}
	With $f(\rho) = (1-\rho)^t \rho$ and $k\ge 1$ we have
	\[
	    f^{(k)}(\rho) = (-1)^k t_{(k-1)} (1-\rho)^{t-k} ((t+1)\rho - k).
	\]
	In particular the first and second derivatives of $f$ are
	\begin{align*}
	    g(\rho) = f'(\rho) &= (1-\rho)^{t-1}(1-(t+1)\rho) \\
	    h(\rho) = f''(\rho) &= t(1-\rho)^{t-2}((t+1)\rho-2) . 
	\end{align*}
\end{lemma}
\begin{proof}
	Straightforward.
\end{proof}

We denote values of $g(\rho)$ by $\phi$. If $\phi$ is a value of $g$ with $\phi>0$ there is exactly one solution of $g(\rho)=\phi$, whereas if $\phi\in (\phimin,0]$ (where $\phimin=g(2/(t+1))$ is the minimum value of $g(\rho)$ on $[0,1]$) then there are exactly two solutions. One of these solutions satisfies $1/(t+1) \le \rho < 2/(t+1)$, and the other satisfies $2/(t+1) < \rho \le 1$.

\begin{corollary}\label{cor:interior}
	Interior critical points for \eqref{opt} are either of the form
$(1/r,1/r,\dots,1/r)$, the Tur\'an solution, or
$(\a,\a,\dots,\a,\b,\b,\dots,\b)$, where
\begin{equation}\label{eq:alpha_beta_bounds}
\frac{1}{t+1} \le \alpha < \frac{2}{t+1} \qquad \text{and} \qquad \frac{2}{t+1} < \beta \le 1.
\end{equation}
and for some
$\phi\in (\phimin,0]$ we have $g(\a)=\phi=g(\b)$, which we will refer to as a
skew solution. In the skew solution case we also require that $a\a+b\b=1$, where
$a$ is the number of $\a$s and $b$ is the number of $\b$s. 
\end{corollary}
 
In this section we will prove the following theorem describing the optimal
solution to \eqref{opt}.


%

\begin{theorem}\label{thm:opt}
Let $r, t \ge 2$. The objective function $F$ is maximized at an interior
critical point.
There are at most two possibilities for this maximizing critical
point. One is the Tur\'an solution. The only other possibility is the skew
solution $(\a,\a,\ldots,\a,\b)$ associated to $a=r-1$ and $b=1$ having
$g(\a)=g(\b)$ largest. If any skew solution exists, then this skew solution
exists.
\end{theorem}

Our approach will be to fix $t$, $a$, and $b$, and consider $\a,\b$ as functions of $\phi$. We are then looking for solutions to 
\[
	\Lab(\phi) = a\a + b\b = 1,
\]
which maximize
\[
	\Fab=af(\a)+bf(\b).
\]
If the context makes it clear, we will omit the subscripts.
We will then consider a critical point $(\a,\a,\ldots,\a,\b,\b,\ldots,\b)$ with $a$ copies of $\a$ and $b$ copies of $\b$ and $\phi=g(\a)=g(\b)$.  If $a<r-1$, we will show that there is a critical point associated to some $\phi'=g(\a')=g(\b')>\phi$ with $a+1$ copies of $\a'$, $b-1$ copies of $\b'$, and a larger value for the objective function.  Thus, we need only consider which critical point associated with the case $a=r-1$ and $b=1$ is best.  We show it is the one with $\phi$ largest.

We begin with some preliminary lemmas.

\begin{lemma}\label{lem:deriv}
	For any real numbers $a,b$ summing to $r$, we have
	\begin{align*}
		\frac{d\Lab}{d\phi}&=\frac{a}{h(\a)}+\frac{b}{h(\b)},\\
		\frac{d\Fab}{d\phi}&=\phi\frac{d\Lab}{d\phi},\text{ and}\\
		\frac{ d^2 \Lab }{d \phi^2}&= \frac{ah'(\a)(h(\b))^3 + bh'(\b)(h(\a))^3}{-(h(\a)h(\b))^3}.
	\end{align*}
\end{lemma}

\begin{proof}
	Since $\phi=g(\a)$, we have that $d\a/d\phi=1/h(\a)$.  Similarly, $d\b/d\phi=1/h(\b)$ from which the first equation follows.  For the second, 
	\[
	    \frac{d F}{d\phi} = a g(\a) \frac{d\a}{d\phi} + b g(\b) \frac{d\b}{d\phi} 
	        = \frac{a \phi}{h(\a)} + \frac{b\phi}{h(\b)} 
	        = \phi \parens[\Big]{ \frac{a}{h(\a)} + \frac{b}{h(\b)}} = \phi \frac{d L}{d \phi}.
	\]
	The third is a straightforward calculation.
\end{proof}

As a consequence, for $\phi_2<\phi_1$, we have
\begin{equation}
    F(\phi_1) - F(\phi_2) = \definite \frac{d F}{d\phi} d\phi{\phi_2}{\phi_1} 
        = \definite \phi \frac{d L}{d \phi} d\phi{\phi_2}{\phi_1} 
        = \Bigl. \phi L \Bigr|_{\phi_2}^{\phi_1} - \definite L d\phi{\phi_2}{\phi_1}.\label{eqn:integ}
\end{equation}

Note that, in the expression for $d^2 L/d\phi^2$ in Lemma~\ref{lem:deriv}, the denominator and the first term on the numerator are always positive and the second term on the numerator is positive provided $\b>3/(t+1)$.  Hence, for $\phi>\pk:=g(3/(t+1))$, we see that $L$ is a convex function of $\phi$.  Our proof will depend on the fact that if $L$ is concave at $\phi$, then $\phi\leq \pk$.  

Now we are ready to begin the proof in earnest. The following sequence of technical lemmas builds our understanding of the relationship between the values of the objective function at the possible internal critical points.

\begin{lemma}\label{lem:mvt}
	If there is a critical point with parameters $\phi$, $a$, and $b$, and $a<r-1$, then there is a critical point associated to $\phi'$, $a+1$, and $b-1$, with $\phi'>\phi$ and $\Fab(\phi)<F_{a+1,b-1}(\phi')$.
\end{lemma}

\begin{proof}
	We have $\Lab(\phi)=1$ and $\a(\phi)<\b(\phi)$, hence $L_{a+1,b-1}<1$.  Also, note that $F_{a+1,b-1}(\phi)=\Fab(\phi)+f(\a)-f(\b)$.  By the Intermediate Value Theorem, there is a root of $L_{a+1,b-1}=1$ between $\phi$ and $0$.  (Note that $L_{a+1,b-1}(0)\geq 2(a+1)/(t+1)+b-1>1$.)  Let $\phi'$ be the smallest such root.  By (\ref{eqn:integ}), we have
	\begin{align*}
		F_{a+1,b-1}(\phi')-F_{a+1,b-1}(\phi)&=\phi' L_{a+1,b-1}(\phi')-\phi L_{a+1,b-1}(\phi)-\definite L_{a+1,b-1}(\rho) d\rho{\phi}{\phi'}\\
		&=\phi'-\phi (1+\a-\b)-\definite L_{a+1,b-1}(\rho) d\rho{\phi}{\phi'}\\
		&=\phi(\b-\a)+(\phi'-\phi)-\definite L_{a+1,b-1}(\rho) d\rho{\phi}{\phi'}\\
		&> \phi(\b-\a),
	\end{align*}
	where the inequality is a consequence of the fact that $L_{a+1,b-1}(\rho)<1$ for $\rho\in (\phi,\phi')$.  Thus,
	\begin{align*}
		F_{a+1,b-1}(\phi')-\Fab(\phi)&=(F_{a+1,b-1}(\phi')-F_{a+1,b-1}(\phi)) +(F_{a+1,b-1}(\phi)-\Fab(\phi))\\
		&>\phi(\b-\a)+f(\a)-f(\b).
	\end{align*}
	So, it suffices to show
	\[
		\frac{f(\b)-f(\a)}{\b-\a}\leq\phi.
	\]
	But by the Mean Value Theorem for some $\rho\in (\a,\b)$, we have 
	\[
		\frac{f(\b)-f(\a)}{\b-\a}=g(\rho).
	\]
	For all $\rho\in (\a,\b)$, we have $g(\rho)<g(\a)=g(\b)=\phi$ and so we are done.
\end{proof}

%

\begin{lemma} \label{lem:sym}
	If $\a<\b\leq 3/(t+1)$ satisfy $g(\a)=\phi=g(\b)$, then
	\[
		\frac{2-(t+1)\a}{(t+1)\b-2} \le 1 .
	\]
\end{lemma}

\begin{proof}
	First note that if $2/(t+1) \le \rho \le 3/(t+1)$ then we have
	\[
	    -h\parens[\Big]{\frac{4}{t+1}-\rho} = t\parens[\Big]{1 + \rho -\frac{4}{t+1}}^{t-2} ((t+1)\rho-2) 
			 \ge t(1-\rho)^{t-2} ((t+1)\rho-2) 
			 = h(\rho),
	\]
	since by hypothesis $2\rho\ge \frac{4}{t+1}$. As a consequence if $2/(t+1) < \beta \le 3/(t+1)$ then 
	\[
		g\parens[\Big]{\frac4{t+1}-\b} \ge g(\b), \text{\ since\ } 
	    g(\b) - g\parens[\Big]{\frac{4}{t+1}-\b} = \int_{\frac{2}{t+1}}^\b h(\rho) + h\parens[\Big]{\frac{4}{t+1}-\rho} \,d\rho \le 0.
	\]
	Now to prove the result we note that since $g\parens[\big]{\frac{4}{t+1}-\b} \ge \phi = g(\b)$ while $g(\a)=\phi$, and $g$ is decreasing on the interval $\parens[\big]{\frac{1}{t+1},\frac{2}{t+1}}$ we must have $\a \ge \frac{4}{t+1}-\b$, which implies the claim. 
\end{proof}

\begin{lemma}\label{lem:ab}
    If $r-1 \ge 2/((t+1)\a(\pk)-1)$, that is, $(t+1)\a(\pk)\geq (r+1)/(r-1)$, and $\phi \le \pk$ then $\dby{L}\phi \le 0$.
\end{lemma}

\begin{proof}
    We have, with $L=L_{r-1,1}$,
    \begin{align*}
        \dby{L}{\phi} &= \frac{r-1}{h(\a)} + \frac{1}{h(\b)} \\
            &= \frac{ (r-1) t(1-\b)^{t-2} ((t+1)\b-2) +  t(1-\a)^{t-2} ((t+1)\a-2)}{h(\a)h(\b)}.
    \end{align*}
    The first term in the numerator is positive and the second is negative. The denominator is negative. Thus $\dby{L}{\phi} \le 0$ precisely if     
    \[
        (r-1)(1-\b)^{t-2}((t+1)\b-2) \ge  (1-\a)^{t-2} (2-(t+1)\a),
    \]
    i.e.,
    \[
        \parens[\big] {\frac{1-\a}{1-\b}}^{t-2} \cdot \frac{2-(t+1)\a}{(t+1)\b-2} = \frac{(t+1)\b-1}{(t+1)\a-1} 
                 \cdot \frac{1-\b}{1-\a} \cdot \frac{2-(t+1)\a}{(t+1)\b-2}\le r-1,
    \]
	where we used the fact that 
	\[
		\parens[\Big]{\frac{1-\a}{1-\b}}^{t-1} = \frac{(t+1)\b-1}{(t+1)\a-1},
	\]
	which is a simple consequence of the fact that $g(\a)=g(\b)$. Both of the ratios $\frac{1-\b}{1-\a}$ and $\frac{2-(t+1)\a}{(t+1)\b-2}$ are at most one; the first because $\b\ge \a$, the second because of Lemma \ref{lem:sym}, so it is sufficient to prove that $\frac{(t+1)\b-1}{(t+1)\a-1} \le r-1$. This fraction is  monotonically increasing in $\phi$ and we have, by hypothesis,
    \[
        \frac{(t+1)\b-1}{(t+1)\a-1} \le \frac{(t+1)\b(\pk)-1}{(t+1)\a(\pk)-1} =
\frac{2}{(t+1)\a(\pk)-1} \le r-1.
    \]
\end{proof}

We require separate arguements for different pairs $(r,t)$. Call a pair $(r,t)$
\begin{itemize}
\item \emph{type A} if $t \le r$.
\item \emph{type B} if $r \ge 7$ and $t \ge r+1$.
\item \emph{type C} if $2 \le r \le 6$ and $t \ge 3r-1$.
\end{itemize}

\begin{lemma}\label{lem:typea}
If $(r,t)$ is a type A pair, then no skew solution exists.
\end{lemma}

\begin{proof}
Recall from (\ref{eq:alpha_beta_bounds}) that $\alpha > 1/(t+1)$. By repeated
application of Lemma~\ref{lem:mvt} we have $a=(r-1)$ and thus
\[ \beta = 1-(r-1)\alpha > \frac{2}{t+1} \implies \alpha < \frac{t-1}{(r-1)(t+1)}. \]
We conclude
\[ \alpha \in \left(\frac{1}{t+1}, \frac{t-1}{(r-1)(t+1)}\right). \]
Of course, if this interval is empty then we have no choices for $\alpha$ and,
due to Lemma \ref{lem:mvt}, no skew solution exists. Thus in order for a
skew solution to exist we must have
\[ \frac{t-1}{(r-1)(t+1)} > \frac{1}{t+1}, \]
which implies $t > r$, so for $t \le r$ no skew solution exists.
\end{proof}

 \begin{lemma}\label{lem:monotone}
	The hypothesis of Lemma~\ref{lem:ab} holds, that is,
	\[
		(t+1)\a(\pk)\geq \frac{r+1}{r-1},
	\]
	for all $(r,t)$ that are type B. 
 \end{lemma}

\begin{proof}
	It is sufficient to prove that, for type B pairs $(r,t)$, we have
	\[
		g\parens[\Big]{\frac{r+1}{(r-1)(t-1)}}\geq g\parens[\Big]{\frac{3}{t+1}}.
	\]
	Noting that both of these expressions are negative, this is equivalent to
	\[
		\frac{2}{r-1}\parens[\Big]{1-\frac{r+1}{(r-1)(t+1)}}^{t-1}\leq 2\parens[\Big]{1-\frac{3}{t+1}}^{t-1},
	\]
	i.e.,
	\[
		\parens[\Big]{\frac{t-\frac{2}{r-1}}{t-2}}^{t-1}
			=\parens[\Big]{1 + \frac{1+\frac{r-3}{r-1}}{t-2}}^{t-1} \leq r-1.
	\]
	The left-hand side converges, as $t$ tends to infinity, to
$\exp(1+(r-3)/(r-1))$.  The smallest $r$ for which $\exp(1+(r-3)/(r-1))\leq r-1$
is $r=6$ so the lemma holds for all type B pairs.
\end{proof}

\begin{corollary}\label{cor:noposderiv}
   For any type B pair $(r,t)$, there is no root of $L=\Lr = 1$ with $\phi\le\pk$ and $\dby{L}{\phi} > 0$.
\end{corollary}

\begin{lemma}
    For type B pairs $(r,t)$, there are at most two roots of $\Lr=1$.
\end{lemma}

\begin{proof}
	Suppose that there are at least three roots of $\Lr=1$, and let $0>\phi_1>\phi_2>\phi_3$ be the three largest. We must have $\dby{L}{\phi}>0$ at $\phi_1$, so by Corollary~\ref{cor:noposderiv}, $\phi_1>\pk$.  As we observed after Lemma~\ref{lem:deriv}, for $L$ to be concave requires $\phi\leq \pk$.  Between $\phi_1$ and $\phi_3$, $L$ must be concave at some point, so $\phi_3<\pk$.  Also, we must have $\dby{L}{\phi}>0$ at $\phi_3$ and this combination is ruled out by Corollary~\ref{cor:noposderiv}.
\end{proof}

\begin{corollary}\label{cor:large}
    For type B pairs $(r,t)$, if $L=\Lr=1$ has multiple solutions, then the one at which $F$ is maximized is the one with $\phi$ largest.
\end{corollary}

\begin{proof}
	By the previous Lemma, there cannot be three roots of $\Lr=1$.  If there are two, say $0>\phi_1>\phi_2$, then by (\ref{eqn:integ}), we have 
	\[
	    F(\phi_1) - F(\phi_2) 
	        = \Bigl. \phi L \Bigr|_{\phi_2}^{\phi_1} - \definite L d\phi{\phi_2}{\phi_1}
			=\parens{\phi_1-\phi_2}-\definite L d\phi{\phi_2}{\phi_1}>0,
	\]
	since $L<1$ for $\phi\in (\phi_2,\phi_1)$.
\end{proof}

\begin{lemma}\label{lem:typec}
If $(r,t)$ is a type C pair, there is exactly one solution to $L=\Lr=1$.
\end{lemma}

\begin{proof}
We use a weaker condition than that of (\ref{eq:alpha_beta_bounds}): if a skew
solution exists, $\alpha \in [0, \frac{2}{t+1})$. We claim that for $r \ge 2$
and $t \ge 3r-1$, $F$ has exactly one critical point in this range. Consider the
derivative
\[ G(\alpha) = F'(\alpha) = (r-1)g(\alpha) - (r-1)g(1-(r-1)\alpha). \]
We start by proving that $G(0) > 0$ and $G(\frac{2}{t+1}) < 0$. We have
\[
G(0) = (r-1)f'(0) - (r-1)f'(1-(r-1)\cdot 0)
	=(r-1)[f'(0)-f'(1)]
	 > 0,
\]
and
\begin{align*}
G\left(\frac{2}{t+1}\right) 
	&= (r-1)\left[-\left(1-\frac{2}{t+1}\right)^{t-1} +
\left(\frac{2(r-1)}{t+1}\right)^{t-1}(t-2(r-1))\right]\\
	&\le (r-1)\left[-\left(\frac{r-1}{r}+\frac{1}{3r}\right)^{t-1} +
\left(\frac{2}{3}\cdot\frac{(r-1)}{r}\right)^{t-1} (t-2(r-1))\right]\\
	&< (r-1)\left[-\left(\frac{r-1}{r}\right)^{t-1} +
\left(\frac{2}{3}\cdot\frac{(r-1)}{r}\right)^{t-1} t\right]\\
	&=
\frac{3}{2}(r-1)\left(\frac{r-1}{r}\right)^{t-1}\left[t\cdot\left(\frac{2}{3}\right)^t-\frac{2}{3}\right],
\end{align*}
where we note that
\[ t \ge 3r-1 \implies -\left(1-\frac{2}{t+1}\right)^{t-1} \le
-\left(\frac{r-1}{r}+\frac{1}{3r}\right)^{t-1}. \]

Each term but the last is positive. As $r \ge 2$ and $t \ge 3r-1$, we may take
$t \ge 5$ and thus the last term is negative. Therefore $G(\frac{2}{t+1}) < 0$
as claimed.

As $G$ is continuous, by the Intermediate Value Theorem it has at least one root
in $[0, \frac{2}{t+1})$. As $G = F'$, a root of $G$ indicates a
critical point of $F$. To prove $F$ has at most one critical point, we start by
show that $G$ is concave up on $[0, \frac{2}{t+1})$. We have
\[ G''(\alpha) = F^{(3)}(\alpha) = (r-1)f^{(3)}(\alpha) - (r-1)^3f^{(3)}(1-(r-1)\alpha). \]
Now
\[ (r-1)f^{(3)}(\alpha) = (r-1)t(t-1)(1-\alpha)^{t-3}(3-(t+1)\alpha) > (r-1)t(t-1)\left(\frac{t-1}{t+1}\right)^{t-3}, \]
as $(1-\alpha) > (t-1)/(t+1)$ and $(t+1)\alpha < 2$ for $\alpha \in [0,\frac{2}{t+1})$.  Also, note that
\begin{align*}
(r-1)^3f^{(3)}(1-(r-1)\alpha) 
	&= (r-1)^3t(t-1)((r-1)\alpha)^{t-3}(t-2-(r-1)(t+1)\alpha)\\
	&< (r-1)^3t(t-1)\left(\frac{2(r-1)}{t+1}\right)^{t-3}(t-2)\\
	&\le (r-1)\left(\frac{t-2}{3}\right)^2t(t-1)\left(\frac{2(t-2)}{3(t+1)}\right)^{t-3}(t-2)\\
	&< (r-1)t(t-1)\left(\frac{t-1}{t+1}\right)^{t-3}\left(\frac{2}{3}\right)^{t-3}\frac{(t-2)^3}{9},
\end{align*}
as $t \ge 3r-1$ implies $r-1 \le \frac{t-2}{3}$. Thus
\[ G''(\alpha) > (r-1)t(t-1)\left(\frac{t-1}{t+1}\right)^{t-3}\left[1-\left(\frac{2}{3}\right)^{t-3}\frac{(t-2)^3}{9}\right]. \]
The last term is positive for $t \ge 19$. One can check that in the remaining cases, when $5 \le t \le 18$ and $2 \le r \le \frac{t+1}{3}$, $G''(\alpha)$ is still positive.

Thus $G$ is concave up on $[0, \frac{2}{t+1})$ and is positive at one
endpoint but negative at the other, so it has at most one zero on that interval.
We conclude $F$ has at most one critical point on $[0, \frac{2}{t+1})$ and thus
at most one skew solution of type $(\alpha, \alpha, \ldots, \alpha, \beta)$
exists. If that solution exists it trivially has $f'(\alpha)$ largest.
\end{proof}


Now we're ready to complete the proof of our main result.

\begin{proof}[Proof of Theorem~\ref{thm:opt}]
First we show that $F$ is not maximized on the boundary of the domain. Suppose, without loss of generality, that
$\rho_1 = 0$ and $\rho_r \neq 0$. Let $\rho_1' = \rho_r' = \frac{\rho_r}{2}$.
Each term of the sum defining $F$, other than the first and last, remains unchanged.
Originally, the first term was $0$ and the last was $\rho_r(1-\rho_r)^t$. Now
each term is $\frac{\rho_r}{2}(1-\frac{\rho_r}{2})^t$, giving a sum of
$\rho_r(1-\frac{\rho_r}{2})^t > \rho_r(1-\rho_r)^t$. We conclude points on the
boundary cannot be maximizers. 

As the domain of $F$ is closed and bounded and $F$ is continuous, it must
achieve its maximum and thus that maximum must occur at an interior point. By
Corollary \ref{cor:interior}, such points only occur at points of the
form $(\alpha, \alpha, \ldots, \alpha, \beta, \beta, \ldots, \beta)$ where
$\alpha < 2/(t+1) < \beta$ and $g(\alpha) = \phi = g(\beta)$ or at points of the
form $(1/r, 1/r, \ldots, 1/r)$.

If there are no critical points of the first type, and in particular if $(r,t)$
is a type A pair, then the only interior critical point is the Tur\'an solution.
In this case, $F$ must attain its maximum here.

Otherwise, there exists at least one skew critical point $(\alpha, \alpha, \ldots,
\alpha, \beta, \beta, \ldots, \beta)$, say with $a$ many $\alpha$s and $b$ many
$\beta$s. Repeatedly applying Lemma~\ref{lem:mvt}, we see that the critical
point at which $F$ attains its maximum is either the Tur\'an solution or the
skew solution associated with $a=r-1$ and $b=1$. For type C pairs,
Lemma~\ref{lem:typec} guarantees there is only one such solution, and for type B pairs
Corollary~\ref{cor:large} assures it is the solution having $\phi$ largest.
There are only finitely many pairs $(r,t)$ with $r,t \ge 2$ that are not of type
A, B, or C and in each case manual inspection shows no such pair has a critical
point other than the Tur\'an solution or skew solution with $a=r-1$ and $b=1$
with $\phi$ largest.
\end{proof}

We know from \cite{BN} that the asymptotic solution to the
$\ex_{S_t}(W,K_{r+1})$ problem is either the Tur\'an solution or some skew
solution. Theorem~\ref{thm:opt} specifies exactly which skew solution is the
possible maximum. In the last theorem, we prove the existence of a sharp
threshold $t^\ast(r)$ where the solution to the $\ex_{S_t}(W,K_{r+1})$ problem
transitions from the Tur\'an solution to the skew solution.

\begin{theorem}
For any $r \ge 2$, there is $t^\ast = t^\ast(r)$ such that for any integer $t$, $F$
is maximized by the Tur\'an solution when $t < t^\ast$ and by a skew solution
for $t \ge t^\ast$.
\end{theorem}

\begin{proof}
Theorem \ref{thm:bn_result}, together with Theorem~\ref{thm:opt} establishes
that for all $r \ge 2$ there are values of $t$ for which a skew solution is
optimal. Therefore for each $r$ there is a smallest integer $\tau$, $r < \tau
\leq r + \sqrt{2r}$ for which a skew solution is optimal.

Let $f(\rho,t) = \rho(1-\rho)^t$ and, for fixed constant $r$, define
\[ F(\rho,t) = (r-1)f(\rho,t)+f(1-(r-1)\rho,t). \]

Let $\alpha \in (\frac{1}{\tau+1}, \frac{1}{r})$ be such that $F(\alpha, \tau)
> F(\frac{1}{r},\tau)$. We will prove $F(\alpha,\tau+1) > F(\frac{1}{r},\tau+1)$ and
thus as $\alpha \in (\frac{1}{\tau+2},\frac{1}{r}) \supseteq (\frac{1}{\tau+1},\frac{1}{r})$,
\[ \max\limits_{\alpha' \in (\frac{1}{\tau+2}, \frac{1}{r})} F(\alpha',\tau+1) \ge
F(\alpha,\tau+1) > F(\tfrac{1}{r},\tau+1). \]
Note that $F(\frac{1}{r},\tau) = (1-\frac{1}{r})^\tau$. Thus we have
\[ F(\alpha,\tau)\cdot (1-\tfrac{1}{r}) > \left(1-\frac{1}{r}\right)^{\tau+1} =
F(\tfrac{1}{r},\tau+1). \]
Next note that
\begin{align*}
F(\alpha,\tau) \cdot (1-\tfrac{1}{r}) &= \left(1-\frac{1}{r}\right)\biggl(
(r-1)\alpha(1-\alpha)^\tau + (1-(r-1)\alpha)((r-1)\alpha)^\tau\biggr)\\
	&= \left(\left(1-\frac{1}{r}\right) - (1-\alpha) + (1-\alpha)\right)
(r-1)\alpha(1-\alpha)^\tau\\
&\qquad\qquad + \left(\frac{r-1}{r} - (r-1)\alpha +
(r-1)\alpha\right)(1-(r-1)\alpha)((r-1)\alpha)^\tau\\
	&= \left(\alpha-\frac{1}{r}\right)(r-1)\alpha(1-\alpha)^\tau +
(r-1)\alpha(1-\alpha)^{\tau+1} \\
	&\qquad \qquad+(r-1)\left(\frac{1}{r} - \alpha\right)(1 -
(r-1)\alpha)((r-1)\alpha)^\tau\\
	&\qquad\qquad +(1-(r-1)\alpha)((r-1)\alpha)^{\tau+1}\\
	&= F(\alpha,\tau+1) - (r-1)\left(\frac{1}{r}-\alpha\right)(f(\alpha,\tau) -
f(1-(r-1)\alpha,\tau)).
\end{align*}
As $\alpha < \frac{1}{r}$, the sign of this second term depends entirely on
$f(\alpha,\tau) - f(1-(r-1)\alpha,\tau)$. Recall that
\[ \frac{\partial}{\partial\rho}f(\rho,t) = (1-\rho)^{t-1}(1-(t+1)\rho) \]
and note that for fixed $t$, $\frac{\partial f}{\partial \rho} < 0$ for $\rho$
such that $\frac{1}{t+1} < \rho < 1$. Thus $f$ is decreasing on this range, and
as
\[ \frac{1}{\tau+1} < \alpha < 1-(r-1)\alpha < 1 \]
because $\alpha < \frac{1}{r}$, we have $f(\alpha,\tau) >
f(1-(r-1)\alpha,\tau)$. We conclude that
\[ F(\alpha,\tau+1) > F(\alpha,\tau+1) -
(r-1)\left(\frac{1}{r}-\alpha\right)(f(\alpha,\tau)-f(1-(r-1)\alpha,\tau)) >
F(\tfrac{1}{r},\tau+1) \]
as claimed.

Finally, by induction we may apply the same reasoning to any $t > \tau$ to see
that if $\alpha$ is the parameter of a skew solution maximizing $F(\alpha,t)$, then
$F(\alpha,t+1) > F(\frac{1}{r},t+1)$ and so the pair $(r,t+1)$ also has a
skew solution. Defining $t^\ast(r)$ to be the $\tau$ corresponding to $r$
completes the proof.
\end{proof}

\begin{remark}
Though Bollob\'as and Nikiforov did not prove the existence of a sharp
threshold, the best known bounds for $t^\ast$ come from
Theorem~\ref{thm:bn_result}: $r < t^\ast \le r+\sqrt{2r}$. By looking at the
proofs of their theorem and Lemma~\ref{lem:typea}, a nearly complete picture of
the optimal graphon for fixed $r$ and varying $t$ emerges. First, while $t \le
r$, no skew solutions exist and thus the Tur\'an solution is optimal. Then skew
solutions emerge but do not immediately beat the Tur\'an solution. When $t$
reaches $t^\ast$, the optimal skew solution overtakes the Tur\'an
solution and will continue to outperform it indefinitely. While obtaining a
precise estimate of $t^\ast$ remains difficult, by $t = r+\lceil\sqrt{2r}\rceil$
there is a skew solution that is not optimal (but whose stars are easy to count)
that outperforms the Tur\'an solution. We believe, based on numeric evidence,
that the true value of $t^\ast$ is closer to the upper bound, but improving either
bound remains an open question.
\end{remark}

%


\bibliographystyle{amsplain}
\bibliography{erdosstoney}

\providecommand{\bysame}{\leavevmode\hbox to3em{\hrulefill}\thinspace}
\providecommand{\MR}{\relax\ifhmode\unskip\space\fi MR }
\providecommand{\MRhref}[2]{%
  \href{http://www.ams.org/mathscinet-getitem?mr=#1}{#2}
}
\providecommand{\href}[2]{#2}
\begin{thebibliography}{10}

\bibitem{AS}
Noga Alon and Clara Shikhelman, \emph{Many {$T$} copies in {$H$}-free graphs},
  J. Combin. Theory Ser. B \textbf{121} (2016), 146--172.

\bibitem{B}
B\'ela Bollob\'as, \emph{On complete subgraphs of different orders}, Math.
  Proc. Cambridge Philos. Soc. \textbf{79} (1976), no.~1, 19--24.

\bibitem{BN}
B\'{e}la Bollob\'{a}s and Vladimir Nikiforov, \emph{Degree powers in graphs
  with forbidden subgraphs}, Electron. J. Combin. \textbf{11} (2004), no.~1,
  Research Paper 42, 8.

\bibitem{CY}
Yair Caro and Raphael Yuster, \emph{A {T}ur\'{a}n type problem concerning the
  powers of the degrees of a graph}, Electron. J. Combin. \textbf{7} (2000),
  Research Paper 47, 14.

\bibitem{C}
Zachary Chase, \emph{The maximum number of triangles in a graph of given
  maximum degree}, Adv. Comb. (2020), Paper No. 10, 5.

\bibitem{EG}
John Engbers and David Galvin, \emph{Counting independent sets of a fixed size
  in graphs with a given minimum degree}, J. Graph Theory \textbf{76} (2014),
  no.~2, 149--168.

\bibitem{E1}
P.~Erd\H{o}s, \emph{On a theorem of {R}ademacher-{T}ur\'an}, Illinois J. Math.
  \textbf{6} (1962), 122--127.

\bibitem{E2}
\bysame, \emph{On the number of complete subgraphs contained in certain
  graphs}, Magyar Tud. Akad. Mat. Kutat\'o Int. K\"ozl. \textbf{7} (1962),
  459--464.

\bibitem{E64}
\bysame, \emph{Extremal problems in graph theory}, Theory of {G}raphs and its
  {A}pplications ({P}roc. {S}ympos. {S}molenice, 1963), Publ. House
  Czechoslovak Acad. Sci., Prague, 1964, pp.~29--36.

\bibitem{ES}
P.~Erd\H{o}s and A.~H. Stone, \emph{On the structure of linear graphs}, Bull.
  Amer. Math. Soc. \textbf{52} (1946), 1087--1091.

\bibitem{E}
Paul Erd\H{o}s, \emph{Some theorems on graphs}, Riveon Lematematika \textbf{9}
  (1955), 13--17.

\bibitem{F}
Zolt\'an F\"uredi, \emph{New asymptotics for bipartite {T}ur\'an numbers}, J.
  Combin. Theory Ser. A \textbf{75} (1996), no.~1, 141--144.

\bibitem{GLS}
Wenying Gan, Po-Shen Loh, and Benny Sudakov, \emph{Maximizing the number of
  independent sets of a fixed size}, Combin. Probab. Comput. \textbf{24}
  (2015), no.~3, 521--527.

\bibitem{GP}
Dániel Gerbner and Balázs Patkós, \emph{Generalized {T}ur\'an problems for
  complete bipartite graphs}, 2021.

\bibitem{GSTZ}
E.~Gy\H{o}ri, N.~Salia, C.~Tompkins, and O.~Zamora, \emph{The maximum number of
  {$P_\ell$} copies in {$P_k$}-free graphs}, Acta Math. Univ. Comenian. (N.S.)
  \textbf{88} (2019), no.~3, 773--778. \MR{4012878}

\bibitem{GPS}
Ervin Gy{\"o}ri, J{\'a}nos Pach, and Mikl{\'o}s Simonovits, \emph{On the
  maximal number of certain subgraphs in $k_r$-free graphs}, Graphs and
  Combinatorics \textbf{7} (1991), no.~1, 31--37 (English).

\bibitem{HP}
Anastasia Halfpap and Cory Palmer, \emph{On supersaturation and stability for
  generalized {T}ur\'an problems}, 2019.

\bibitem{LM}
Bernard Lidický and Kyle Murphy, \emph{Maximizing five-cycles in $k_r$-free
  graphs}, 2021.

\bibitem{LS}
L.~Lov\'asz and M.~Simonovits, \emph{On the number of complete subgraphs of a
  graph},  (1976), 431--441. Congressus Numerantium, No. XV.

\bibitem{LS1}
\bysame, \emph{On the number of complete subgraphs of a graph. {II}}, Studies
  in pure mathematics, Birkh\"auser, Basel, 1983, pp.~459--495.

\bibitem{LL}
L\'aszl\'o Lov\'asz, \emph{Combinatorial problems and exercises}, second ed.,
  AMS Chelsea Publishing, Providence, RI, 2007.

\bibitem{L}
Ruth Luo, \emph{The maximum number of cliques in graphs without long cycles},
  J. Combin. Theory Ser. B \textbf{128} (2018), 219--226. \MR{3725194}

\bibitem{MM}
J.~W. Moon and L.~Moser, \emph{On a problem of {T}ur\'an}, Magyar Tud. Akad.
  Mat. Kutat\'o Int. K\"ozl. \textbf{7} (1962), 283--286.

\bibitem{N1}
V.~Nikiforov, \emph{The number of cliques in graphs of given order and size},
  Trans. Amer. Math. Soc. \textbf{363} (2011), no.~3, 1599--1618.

\bibitem{N}
Vladimir Nikiforov, \emph{Graphs with many {$r$}-cliques have large complete
  {$r$}-partite subgraphs}, Bull. Lond. Math. Soc. \textbf{40} (2008), no.~1,
  23--25.

\bibitem{PY}
Oleg Pikhurko and Zelealem~B. Yilma, \emph{Supersaturation problem for
  color-critical graphs}, J. Combin. Theory Ser. B \textbf{123} (2017),
  148--185.

\bibitem{R}
Hans Rademacher, unpublished (1941).

\bibitem{Rz}
Alexander~A. Razborov, \emph{On the minimal density of triangles in graphs},
  Combin. Probab. Comput. \textbf{17} (2008), no.~4, 603--618.

\bibitem{Rei}
Christian Reiher, \emph{The clique density theorem}, Ann. of Math. (2)
  \textbf{184} (2016), no.~3, 683--707.

\bibitem{T}
Paul Tur\'an, \emph{Eine {E}xtremalaufgabe aus der {G}raphentheorie}, Mat. Fiz.
  Lapok \textbf{48} (1941), 436--452.

\bibitem{W}
David~R. Wood, \emph{On the maximum number of cliques in a graph}, Graphs
  Combin. \textbf{23} (2007), no.~3, 337--352.

\bibitem{Z}
A.~A. Zykov, \emph{On some properties of linear complexes}, Mat. Sbornik N.S.
  \textbf{24(66)} (1949), 163--188.

\end{thebibliography}

\end{document}